\documentclass[10pt,oneside,a4paper,reqno]{amsart}  

\usepackage{amssymb,graphicx,xcolor}
\usepackage[colorlinks=true, linkcolor=magenta, citecolor=cyan, urlcolor=blue]{hyperref}
\usepackage{bbm}

\numberwithin{equation}{section}\newtheorem{theorem}{Theorem}[section]
\newtheorem{lemma}[theorem]{Lemma}

\newtheorem{proposition}[theorem]{Proposition}\theoremstyle{remark}
\newtheorem{remark}{Remark}[section]
\theoremstyle{definition}
\newtheorem{definition}[theorem]{Definition}

\def\ni{\noindent}

\newcommand{\p}{\widetilde{p}}
\newcommand{\q}{\widetilde{q}}

\newcommand{\Lqqtilde}{L^{q}_{|x|}L^{\widetilde{q}}_{\theta}}
\newcommand{\Lpptilde}{L^{p}_{|x|}L^{\widetilde{p}}_{\theta}}

\newcommand{\Lpptildealphazero}{L^{p_{0}}_{|x|^{\alpha_{0} p_{0}}d|x|}L^{\widetilde{p}_{0}}_{\theta}}

\newcommand{\Lqqtildebetazero}{L^{q_{0}}_{|x|^{\beta_{0} q_{0}}d|x|}L^{\widetilde{q}_{0}}_{\theta}}
\newcommand{\Lpptildealphauno}{L^{p_{1}}_{|x|^{\alpha_{1} p_{1}}d|x|}L^{\widetilde{p}_{1}}_{\theta}}
\newcommand{\Lqqtildebetauno}{L^{q_{1}}_{|x|^{\beta_{1} q_{1}}d|x|}L^{\widetilde{q}_{1}}_{\theta}}
\newcommand{\Lpptildealphaxi}{L^{p_{\xi}}_{|x|^{\alpha_{\xi} p_{\xi}}d|x|}L^{\widetilde{p}_{\xi}}_{\theta}}
\newcommand{\Lqqtildebetaxi}{L^{q_{\xi}}_{|x|^{\beta_{\xi} q_{\xi}}d|x|}L^{\widetilde{q}_{\xi}}_{\theta}}

\newcommand{\Lrqqtilde}{L^{r}_{t}L^{q}_{|x|}L^{\widetilde{q}}_{\theta}}
\newcommand{\Lspptilde}{L^{s}_{t}L^{p}_{|x|}L^{\widetilde{p}}_{\theta}}
\newcommand{\Rn}{\mathbb{R}^{n}}
\newcommand{\Rpiu}{\mathbb{R}^{+}}

\title[Angular Integrability and Navier--Stokes equation]
{Regularity criteria with angular integrability for the Navier--Stokes equation}
\date{\today}    

\author{Renato Luc\`a}
\address{Instituto de Ciencias Matematicas, Consejo Superior de Investigaciones Cientificas, Madrid, 28049, Spain.
Supported by the ERC grant 277778 and MINECO grant SEV-2011-0087 (Spain).}
\email{renato.luca@icmat.es}

\begin{document}

 \begin{abstract}
We give new {\it a priori} assumptions on weak solutions of the Navier--Stokes equation so as to be able to conclude that they are smooth. The regularity criteria are given in terms of mixed radial-angular weighted Lebesgue space norms.    
\end{abstract}

\maketitle

\section{Introduction and main results}
We consider the Cauchy problem 
on $(0,T) \times \mathbb{R}^{n}$

\begin{equation}\label{CauchyNS}
\left \{
\begin{array}{rcl}
\partial_{t}u + (u \cdot \nabla) u  -\Delta u & = & -\nabla P  \\
\nabla \cdot u & = & 0 \\
u(x,0) & = & u_{0}(x). 
\end{array}\right.
\end{equation}

It describes the motion of a viscous 
incompressible fluid in the absence of external forces,
where
$u$ is the velocity field and $P$ is the pressure. 

The first equation is the Newton law while the 
second follows by the incompressibility of the fluid. 
In order to require incompressibility at time $t=0$ 
it is necessary to restrict to initial data $u_{0}$ such 
that $\nabla \cdot u_{0}=0$.

We shall use the same notation for the norm of scalar, 
vector or tensor quantities, for instance: 
$$
\textstyle
\| P \|^{2}_{L^{2}} := \int P^{2} \ dx, 
\qquad
\| u \|^{2}_{L^{2}} := \int \sum_{i=1}^{n} u_{i}^{2} \ dx,
\qquad
\| \nabla u \|^{2}_{L^{2}} := \int \sum_{i,j=1}^{n}(\partial_{i} u_{j})^{2} \ dx
$$ 
and we often write simply $u \in L^{2}(\mathbb{R}^{n})$ instead of $u \in [{L^{2}(\mathbb{R}^{n})}]^{n}$.

The well-posedness of (\ref{CauchyNS}) is still open even if many partial results have been obtained. In \cite{Hopf, Ler} the authors proved global existence of weak solutions for initial data in  $L^{2}$ but a 
satisfactory well-posedness theory is basically developed only in the 
case of small initial data or data with a peculiar geometric structure.

In this scenario it is useful to establish {\it a priori} conditions under which uniqueness and regularity of the weak solutions are guaranteed. Results of this kind are usually called regularity criteria.

In this paper we focus on some classical regularity criteria \cite{CKN, Ser, Sohr, Struwe} and their extension to the setting of weighted Lebesgue spaces \cite{YongZhou}. In particular we show how the results in \cite{YongZhou} can be improved under the hypothesis of additional angular integrability.

The regularity is basically ensured by boundedness assumptions
on quantities like $u, \nabla u, \nabla \times u$ in  
suitable critical spaces. 
A simple regularity criterion is for instance 
\begin{equation}\label{SerrinNorm}
\|u\|_{L^{s}_{T}L^{p}_{x}} := \left( \int_{0}^{T} \left( \int_{\mathbb{R}^{n}} |u(t,x)|^{p} \ dx \right)^{\frac{s}{p}} \ dt \right)^{\frac{1}{s}} < +\infty, 
\qquad 
\frac{2}{s} + \frac{n}{p} \leq 1.
\end{equation}     
Notice that in the endpoint case (\ref{SerrinNorm}) is invariant with respect to
\begin{equation}\label{Scaling}
u(t,x) \rightarrow \lambda u (\lambda^{2} t, \lambda x ),
\end{equation}
that is the natural scaling of (\ref{CauchyNS}). 
In \cite{Ser} smoothness in space variables has been obtained in the case $\frac{2}{s} + \frac{n}{p} < 1$, while the endpoints have been fixed in \cite{Esc,Giga, Sohr, Struwe, WW}. 
We recall the following

\begin{definition}[\cite{CKN}]
We say that a point $(\bar{t},\bar{x}) \in (0,T)
\times \mathbb{R}^{3}$ is 
\emph{regular} for a solution $u(t,x)$ of
(\ref{CauchyNS}) if $u$ is essentially bounded on a neighbourhood of
$(\bar{t},\bar{x})$.
(In this case one can prove that $u(t,x)$ is smooth near
$(\bar{t},\bar{x})$, see for instance \cite{Ser}). 
We say that a set is \emph{regular} if all its points
are regular.
\end{definition}

Let us also recall that  $(0,T) \times \mathbb{R}^{n}$
is regular provided that (\ref{SerrinNorm}) is satisfied
with $2/s + n/p = 1$
(see for instance \cite{Sohr, Struwe}).

Then we focus on the weighted norm approach:

\begin{theorem}[\cite{YongZhou}]\label{YZTheoremSS}
Let $n \geq 3$ and $u_{0} \in L^{2}(\mathbb{R}^{n})$ be  
a divergence free vector field. 
Let then $u$ be a weak solution of (\ref{CauchyNS}) and
$\bar{x} \in \mathbb{R}^{n}$ such that
\begin{equation}
\||x-\bar{x}|^{1 - \frac{n}{2}}u_{0}\|_{L^{2}_{x}} < +\infty,
\end{equation}
 
\begin{equation}\label{MainAssumptYZ}
\||x-\bar{x}|^{\alpha}u(x,t)\|_{L^{s}_{T}L^{p}_{x}} < +\infty,
\end{equation}
with
\begin{equation}
\begin{array}{ll}
\frac{2}{s} + \frac{n}{p}=1 - \alpha, &  -1 \leq \alpha <1 \\
  &  \\
\frac{2}{1-\alpha} < s < +\infty, & \frac{n}{1-\alpha} < p< +\infty ;
\end{array}
\end{equation}
or
\begin{equation}
\||x-\bar{x}|^{\alpha}u(x,t)\|_{L^{\frac{2}{1-\alpha}}_{T}L^{\infty}_{x}} < +\infty, \quad  -1 < \alpha < 1;
\end{equation}
or
\begin{equation}
\sup_{t \in (0,T)}\||x-\bar{x}|^{\alpha}u(x,t)\|_{L^{\frac{n}{1-\alpha}}_{x}} < \varepsilon, \quad  -1 \leq \alpha \leq 1;
\end{equation}
with $\varepsilon$ sufficiently small. Then $(0,T) \times \{ \bar{x} \} $ is a regular set.
\end{theorem}  

\begin{remark}
The condition $\frac{2}{s} + \frac{n}{p}=1 - \alpha$ makes the
norm
$$
\||x-\bar{x}|^{\alpha}u(x,t)\|_{L^{s}_{T}L^{p}_{x}}   
$$
scaling invariant with respect to
$$
u(t, x-\bar{x}) \rightarrow \lambda u(\lambda^{2} t, \lambda (x-\bar{x}) ).
$$ 
\end{remark}

Our goal is to point out the local aspect of Theorem \ref{YZTheoremSS}: 
for each $t \in (0,T)$ there is a neighborhood\footnote{We mean a neighborhood in the space variables for each fixed time.} 
$\Omega_{t,\bar{x}}$ of $\bar{x}$ such that $u$ 
is smooth in $\{ t \} \times \Omega_{t,\bar{x}}$.

The restriction to a neighborhood of $\bar{x}$ can be 
heuristically explained in the case $\alpha < 0$: the weight
morally localizes the norm around $\bar{x}$ and a loss of 
information at infinity occurs. 

We shall show how to recover this information by a 
suitable amount of angular regularity (if $\alpha < 0$) and 
how to do the same in the case $\alpha >0$ even if 
weaker angular regularity is assumed. 


By translations it is possible to restrict to the case $\bar{x} =0$. 
All the following results
are of course true provided with the norms and weights 
centered at $\bar{x} \neq 0$.

In order to quantify precisely our notion of angular 
regularity we define the norms
\begin{equation}\label{MixedNorms}
\begin{array}{lcl}
  \|f\|_{L^{p}_{|x|}L^{\p}_{\theta}}& := &
  \left(
    \int_{0}^{+\infty}
    \|f(\rho\ \cdot\ )\|^{p}_{L^{\p}(\mathbb{S}^{n-1})}
    \rho^{n-1}d \rho
  \right)^{\frac1p}, \\
  \|f\|_{L^{\infty}_{|x|}L^{\p}_{\theta}}& := &
  \sup_{\rho>0}\|f(\rho\ \cdot\ )\|_{L^{\p}(\mathbb{S}^{n-1})}.
\end{array}
\end{equation}
If $p=\p$ the norms reduce to the usual $L^{p}$ norms
\begin{equation*}
  \|u\|_{L^{p}_{|x|}L^{p}_{\theta}} =
  \|u\|_{L^{p}(\mathbb{R}^{n})},
\end{equation*}
while for radial functions the value of $\p$ is irrelevant
\begin{equation*}
  \text{$u$ radial}\quad\implies\quad 
  \|u\|_{L^{p}L^{\p}}\simeq \|u\|_{L^{p}(\mathbb{R}^{n})}
  \quad \forall p,\p\in[1,\infty].
\end{equation*}
Notice also that the norms (ignoring the constants) are increasing in $\p$.

The idea of distinguishing radial and angular directions 
is not new and
has proved successful in the context of
Strichartz estimates and dispersive equations
(see for instance
\cite{DanconaCacciafesta11-a}, \cite{ChoOzawa09-a},
\cite{FangWang08-a},
\cite{MachiharaNakamuraNakanishi05-a},
\cite{Rogers}
\cite{Sterbenz05-a}).

We also notice that the mixed angular-radial norms 
have the same scaling of their classical counterparts, in fact 
$$
\| |x|^{\alpha} u(t,x)\|_{L^{s}_{T}L^{p}_{|x|}L^{\p}_{\theta}} 
$$
is invariant with respect to
$$
u(t,x)\rightarrow  \lambda u (\lambda^{2} t, \lambda x),
$$   
provided that $\frac{2}{s} + \frac{n}{p} = 1-\alpha$.

We obtain new values $\p_{G},\p_{L}$ for 
the angular integrability such that 
global and local\footnote {Here and in the following we mean global and local in space.} 
regularity are, respectively, achieved:

\begin{equation}\label{PtildeDef}
\p_{L} := \left \{
\begin{array}{lcr}
 \frac{2(n-1)p}{(2 \alpha +1)p +2(n-1)}  &  \mbox{if} & -\frac{1}{2} \leq \alpha < 0  \\
 && \\
 \frac{2(n-1)p}{p +2(n-1)}  &  \mbox{if} & 0 \leq \alpha < 1, 
\end{array}\right.
\end{equation}
\begin{equation}\label{PtildeGDef}
\p_{G} := \left \{
\begin{array}{lcr}
 \max \left( 2n, \frac{(n-1)p}{\alpha p + n-1} \right)  &  \mbox{if} & \frac{1-n}{2} < \alpha < 0  \\
 && \\
  \frac{(n-1)p}{\alpha p + n-1}    &  \mbox{if} & 0 \leq \alpha < \frac{1}{•2}.
\end{array}\right.
\end{equation}
Notice that neither the quantities are increasing in $\alpha$,  
$$
\p_{L} < \p_{G}, \quad \mbox{if} \quad \alpha <  1/2, \qquad
\p_{L} = \p_{G} \quad \mbox{if} \quad \alpha = 1/2;
$$
and\footnote{Notice that $\p_{L} = p$ in the endpoint case $\alpha = -1/2$.}
\begin{equation}\label{pLpGalphaPositivo}
\p_{L} \leq p < \p_{G}, \qquad \mbox{if} \quad \alpha <0, 
\end{equation}
\begin{equation}\label{pLpGalphaNegativo}
\p_{L} < \p_{G} < p, \qquad \mbox{if} \quad \alpha > 0;
\end{equation}
this is in fact consistent with the previous heuristic. 
For simplcity we state our results in the case of Schwartz initial data.
In Section 4 we show how to refine this assumption.



\begin{theorem}\label{OurYZTheorem}
Let $n \geq 3$ and $u_{0}$ be a divergence free vector 
field with each component in the Schwartz class. Let also $u$ be a weak solution
of (\ref{CauchyNS}) satisfying (\ref{HPEquiv}). Then $(0,T) \times \Rn$ is a regular set provided that

\begin{equation}\label{OurYZCondition0}
\alpha \in ((1-n)/2,0), 
\quad \max \left( 2, \frac{n}{1- \alpha} \right) < p \leq \frac{1-n}{\alpha},
\quad or \quad p=2,
\end{equation}
and

\begin{equation}\label{OurYZBound1}
\| |x|^{\alpha} u \|_{L^{s}_{T}L^{p}_{|x|}L^{\p}_{\theta}} < +\infty,
\end{equation}
with
\begin{equation}\label{OurYZCondition1}
\frac{2}{s}+ \frac{n}{p} = 1-\alpha,
\end{equation} 

\begin{equation}\label{OurYZCondition2} 
\max \left( 2, \frac{2}{1-\alpha} \right) < s < +\infty, \quad or \quad
s = \frac{2}{1-\alpha},
\end{equation}

\begin{equation}\label{OurYZCondition3}
\p \geq \p_{G} := \max \left(2n, \frac{(n-1)p}{\alpha p +n -1}\right);
\end{equation}
or  
\begin{equation}\label{OurYZCondition0bis}
\alpha \in [0,1/2), \quad 
2n < p \leq +\infty,
\end{equation}
and 
\begin{equation}\label{OurYZBound1bis}
\| |x|^{\alpha} u \|_{L^{s}_{T}L^{p}_{|x|}L^{\p}_{\theta}} < +\infty,
\end{equation}
with
\begin{equation}\label{OurYZCondition1bis}
\frac{2}{s}+ \frac{n}{p} = 1-\alpha,
\end{equation} 

\begin{equation}\label{OurYZCondition2bis} 
\frac{2}{1-\alpha} \leq s < +\infty,
\end{equation}

\begin{equation}\label{OurYZCondition3bis}
\p \geq \p_{G} := \frac{(n-1)p}{\alpha p +n -1}.
\end{equation}

\end{theorem}

\begin{remark} 

Let us point out again that the main information of the 
Theorem is contained in the assumptions 
(\ref{OurYZCondition3}, \ref{OurYZCondition3bis}), {\it i.e.} 
the angular integrability necessary in order to 
get a global regularity result.

It turns out by relations (\ref{pLpGalphaPositivo}, \ref{pLpGalphaNegativo}) 
that in the case of negative weights additional 
angular integrability ($\p_{G}>p$) is necessary in order 
to get global regularity. On the other hand if we 
consider $|x|^{\alpha}, \alpha > 0$ then the additional information at infinity 
allows to get global regularity even for weaker angular integrability ($\p_{G} < p$). 
\end{remark}

\begin{remark} Notice that:
\begin{itemize}
\item Our method misses the endpoint $s=+\infty$; 

\item If $n > 3$ we get a gain in the negative values 
of $\alpha$ with respect to Theorem \ref{YZTheoremSS}. 
We have in fact $\frac{1-n}{2} < \alpha$ instead 
of $-1 \leq \alpha$. This is also more satisfactory 
because exhibits a dependence on the dimension. 
We have, on the other hand, a loss in the 
positive values of $\alpha$, {\it i.e.} $\alpha < \frac{1}{2}$ instead 
of $\alpha < 1$.
\end{itemize}
\end{remark}

\begin{theorem}\label{OurYZTheoremLoc}
Let $n \geq 3$ and $u_{0}$ be a divergence free vector field with each  component in the Schwartz class. Let also $u$ be a weak solution of (\ref{CauchyNS})  satisfying (\ref{HPEquiv}). Then 
$(0,T) \times \{ 0\}$ is a regular set provided that
\begin{equation}\label{OurYZCondition0Loc}
\alpha \in [-1/2,0), \quad 
n < p \leq +\infty,
\end{equation} 
and
\begin{equation}\label{OurYZBound1Loc}
\| |x|^{\alpha} u \|_{L^{s}_{T}L^{p}_{|x|}L^{\p}_{\theta}} < +\infty,
\end{equation}
with
\begin{equation}\label{OurYZCondition1Loc}
\frac{2}{s}+ \frac{n}{p} = 1-\alpha,
\end{equation} 

\begin{equation}\label{OurYZCondition2Loc} 
\max \left( 2, \frac{2}{1-\alpha} \right) < s < +\infty,
\quad or \quad s = \frac{2}{1-\alpha}, 
\end{equation}

\begin{equation}\label{OurYZCondition3Loc}
\p \geq \p_{L} := \frac{2(n-1)p}{(2\alpha +1)p + 2(n-1)};
\end{equation}
or
\begin{equation}\label{OurYZCondition0bisLoc}
\alpha \in [0,1), \quad 
\frac{n}{1 - \alpha} < p \leq +\infty,   
\end{equation}
and 
\begin{equation}\label{OurYZBound1bisLoc}
\| |x|^{\alpha} u \|_{L^{s}_{T}L^{p}_{|x|}L^{\p}_{\theta}} < +\infty,
\end{equation}
with
\begin{equation}\label{OurYZCondition1bisLoc}
\frac{2}{s}+ \frac{n}{p} = 1-\alpha,
\end{equation} 

\begin{equation}\label{OurYZCondition2bisLoc} 
\frac{2}{1-\alpha} \leq s < +\infty, 
\end{equation}

\begin{equation}\label{OurYZCondition3bisLoc}
\p > \p_{L} := \frac{2(n-1)p}{p + 2(n-1)}.
\end{equation}
\end{theorem}

\begin{remark} 
Notice that: 
\begin{itemize}
\item Our main assumption is actually weaker that (\ref{MainAssumptYZ}) 
because $\p_{L} < p$ ( $\p_{L} = p \ \mbox{if} \ \alpha = -1/2$);
 
\item We have a loss in the negative values of $\alpha$ with 
respect to \ref{YZTheoremSS}. We assume in fact $-\frac{1}{2} \leq \alpha$ 
instead of $-1 \leq \alpha$.
\end{itemize}
\end{remark}

It is interesting to compare this results with 
the regularity criteria obtained by working in parabolic
Morrey spaces \cite{Kuk1, Kuk2, Tay}. Consider the norms

\begin{equation*}
\| u \|_{L^{p}_{\lambda}((0,T) \times \Rn)} := \sup_{ \bar{t} \in (0,T), \bar{x} \in \Rn} \sup_{r>0} 
\frac{1}{r^{ \lambda / p}}  \| u \|_{L^{p}(Q_{r}(\bar{t}, \bar{x}))}, 
\end{equation*}
where $Q_{r}(\bar{t}, \bar{x})$ is the parabolic cylinder 
of radius $r$ and centered in $(\bar{t}, \bar{x})$  
\begin{equation*}
Q_{r}(\bar{t}, \bar{x}) := B_{r}(\bar{x}) \times (\bar{t} - r^{2}, \bar{t} + r^{2})
\end{equation*}
and focus on the formal corrispondence 

\begin{equation*}
\| u \|_{L^{p}_{\lambda}((0,T) \times \Rn)} \leftrightarrow 
\sup_{\bar{x} \in \Rn}\| |x- \bar{x}|^{ - \lambda / p} u \|_{L^{p}_{T}L^{p}_{x}}; 
\end{equation*}
Since $ \||x-\bar{x}|^{- \lambda / p} u \|_{L^{p}_{T}L^{p}_{x}} \geq 
  \sup_{\bar{t} \in (0,T)}  \sup_{r>0} 
\frac{1}{r^{ \lambda / p}}  \| u \|_{L^{p}(Q_{r}(\bar{x}, \bar{t})}$ it is 
clear that boundedness assumptions in weighted spaces 
are stronger then their counterpart in Morrey spaces. 
This is heuristically because in the first case
the weights provide a 
residual information even for large $|x|$. 
As we have observed this information and 
angular integrability hypotesis provide a  
a more satisfactory ragularity theory.

We exploit again this viewpoint through a really 
interesting example, {\it i.e.} the weighted counterpart 
of the following

\begin{theorem}[\cite{CKN}]\label{CKNTheorem}
Let $n=3$ and $u$ be a suitable weak solution of (\ref{CauchyNS}).
There is an absolute constant $\varepsilon$ such that if 
\begin{equation}\label{MorreyCKN}
\limsup_{r \rightarrow 0} 
\frac{1}{r^{2}}\int_{Q^{*}_{r}(\bar{t},0)} |u|^{3} +  |p|^{3/2}  \leq \varepsilon,  
\end{equation}  
where 
$$
Q^{*}_{r}(\bar{t},0)
:= \left\{ (\tau, y) : \quad |y| < r, \ 
\bar{t} - 7/8 r^{2} < \tau < \bar{t} + 1/8 r^{2} \right\};
$$
then 
$(\bar{t}, 0)$ is a regular point. 
\end{theorem}
We focus on the condition
\begin{equation}\label{OurMorreyCKN}
\| |x|^{-2/3} u \|_{L^{3}_{T}L^{3}_{|x|}L^{3}_{\theta}} < \infty.  
\end{equation}
A little work is necessary in order to show that (\ref{OurMorreyCKN}) 
is actually stronger than (\ref{MorreyCKN}). We just sketch the argument that is 
classical in the context of the Navier--Stokes theory. 
At first notice that
the pressure can be recovered by $u$ through\footnote{See also the next Section.}
$$
P = \sum_{i,j=1}^{3}R_{i}R_{j}u_{i}u_{j},
$$
where $R_{i}$ is the Riesz transform in the $i$-th direction. 
So the second term in (\ref{MorreyCKN}) can be bounded 
by using the Calderon-Zygmund inequality (see \cite{Stein})
\begin{equation*}
\| |x|^{\delta} P \|_{L^{r}(\mathbb{R}^{3})} \leq
C  \| |x|^{\delta} |u|^{2}\|_{L^{r}(\mathbb{R}^{3})},  
\end{equation*}
$$ 
r \in (1,\infty), \quad -\frac{3}{r} < \delta < 1 - \frac{3}{r},
$$
with the choice $(\delta,r) = (-4/3, 3/2)$. 
Then the smallness assumption in (\ref{MorreyCKN}) easily follows by (\ref{OurMorreyCKN})
provided that $T > \bar{t}$:
\begin{eqnarray}
 \limsup_{r\rightarrow 0} \frac{1}{r^{2}}\int_{Q^{*}_{r}(\bar{t},0)} 
|u|^{3}
&=& \limsup_{r\rightarrow 0} \frac{1}{r^{2}}
\int_{\bar{t} - 7/8 r^{2}}^{\bar{t} + 1/8 r^{2}}\int_{B(0,r)} 
|u|^{3} \nonumber   \\
& \leq & \limsup_{r\rightarrow 0} 
\int_{\bar{t} - 7/8 r^{2}}^{\bar{t} + 1/8 r^{2}}\int_{B(0,r)} 
|x|^{-2}|u|^{3} \nonumber   \\
& \leq & \limsup_{r\rightarrow 0}  
\int_{\bar{t} - 7/8 r^{2}}^{\bar{t}+ 1/8 r^{2}}\int_{\Rn} 
|x|^{-2}|u|^{3} =0. \nonumber 
\end{eqnarray}
Then also notice that $(\alpha, p, s)= (-2/3, 3, 3)$ is an admissible 
choice of indexes in Theorem \ref{YZTheoremSS}.

Theorems \ref{OurYZTheorem}, 
\ref{OurYZTheoremLoc} suggest
that it is possible to 

\begin{enumerate}
    \item get global regularity (in $(0,T) \times \Rn$) 
    by a suitable amount of angular integrability in (\ref{OurMorreyCKN});  \\
    \item get regularity in $(0,T) \times \{ 0 \}$ even by
    weaker angular integrability in (\ref{OurMorreyCKN}). 
\end{enumerate}  
The first point is achieved by applying Theorem \ref{OurYZTheorem} with 
$$
(\alpha,s, p, \p) = (-2/3, 3, 3, +\infty),
$$ 
{\it i.e.} by assuming 
\begin{equation*}
\| |x|^{-2/3} u \|_{L^{3}_{T}L^{3}_{|x|}L^{\infty}_{\theta}} < +\infty ;
\end{equation*}
notice that in this case the indexes satisfy the endpoint 
relation $p=\frac{1-n}{\alpha}$ so we have to require $L^{\infty}$ 
boundedness in the angular direction.

Otherwise it is interesting to notice that Theorem \ref{OurYZTheoremLoc} 
can not give a positive answer to the second point because the value 
$\alpha = - 2/3$ is not permitted. This is of course due to our method and 
in particular to the fact that we never work directly with the energy estimate 
as in \cite{YongZhou}. Otherwise a more direct proof of Theorem \ref{OurYZTheoremLoc} 
requires a delicate analysis of the properies of the Riesz transform 
in mixed radial-angular spaces. This is also a topic of independent interest 
and we hope to reexamine it in a future work.      

The rest of the paper is organized as follows: in the second Section we recall the well known integral formulation of (\ref{CauchyNS});
in the third Section we prove time decay estimates for the heat and Oseen kernels in weighted $L^{p}_{|x|}L^{\p}_{\theta}$ spaces; 
in the fourth Section we prove the main Theorems.

\section{Integral formulation of the problem}

We recall the integral formulation of the Navier--Stokes problem. 
By taking the divergence of the first equation in (\ref{CauchyNS})
and by using the incompressibility:

\begin{eqnarray}\label{PressureRie}
- \Delta P &=&  \sum_{i=1}^{n} \partial_{i} \sum_{j=1}^{n}u_{j}\partial_{j}u_{i}  \\ 
&=& \sum_{i,j=1}^{n}\partial_{i}\partial_{j}(u_{i}u_{j}), 
\end{eqnarray}
so $P$ can be, at least formally, recovered by $u$ through
\begin{equation}
\textstyle
P = - \Delta^{-1} \sum_{i,j=1}^{n}\partial_{i}\partial_{j}(u_{i}u_{j}).
\end{equation} 
Thus (\ref{CauchyNS}) becomes
\begin{equation}\label{CauchyNSInt}
\left \{
\begin{array}{rclcl}
 u & = & e^{t\Delta}u_{0} - \int_{0}^{t}e^{(t-s)\Delta}\mathbb{P} \nabla \cdot (u \otimes u) ds &\quad \mbox{in} \quad & [0,T)\times \mathbb{R}^{n}  \\
\nabla \cdot u & = & 0 &\quad \mbox{in} \quad & [0,T)\times \mathbb{R}^{n}, 
\end{array}\right.
\end{equation}
where $(u\otimes u)_{i,j} := u_{i}u_{j}$ and $\mathbb{P}$ is formally defined by
\begin{equation}\label{LerProj}
\mathbb{P} f :=  f - \nabla \Delta^{-1} (\nabla \cdot f).
\end{equation}
This operator is a really useful tool in 
the study of the Navier--Stokes problem. 
It is actually a projection on the subspace 
of the divergence free vector fields 
($\mathbb{P}f=f \Leftrightarrow \nabla \cdot f =0$).
If $f \in [L^{2}(\mathbb{R}^{n})]^{n}$ then $\mathbb{P}$
is rigorously defined by  
$$
\mathbb{P}f := f  +  R \otimes R \ f,
$$
where $R$ is the vector of 
the Riesz transformations. 
On the other hand $\mathbb{P}$ can be defined on larger 
Banach spaces as a Calderon-Zygmund operator. 
Furthermore we are basically interested in the operator $\mathbb{P}(\nabla \cdot \ \ )$
that, thanks to the differentiation, can be 
actually defined on 
$[L^{1}_{uloc}(\Rn)]^{n \times n}$, 
{\it i.e.} the space of uniformly locally integrable functions 
(see \cite{Lem} for further details).

Now we focus on some properties of the Oseen kernel. At first we need the following
\begin{lemma}[\cite{Lem}]\label{OseenKernelTheorem}
Let $1\leq i,j \leq n$. The operator 
$\Delta^{-1} \sum_{j=1}^{n}\partial_{i}\partial_{j}e^{t\Delta}$ 
is a convolution operator $\sum_{j=1}^{n}O_{i,j}*f_{j}$ with
$$
O_{i,j}(t,x) := \frac{1}{t^{\frac{n}{2}}}o_{i,j}\left( \frac{x}{\sqrt{t}} \right)
$$
and for each multi-index $\eta$
$$
o_{i,j} \in C^{\infty}(\Rn), \qquad
(1+|x|)^{n+|\eta|}\partial^{\eta}o_{i,j} \in L^{\infty}(\Rn).
$$

\end{lemma}
This is the main technical tool necessary in order to study the properties of 
$ e^{t\Delta}\mathbb{P}(\nabla  \cdot \ \ ) $,
it holds in fact the following
\begin{proposition}[\cite{Lem}]\label{OseenDecayFinale}
Let $1\leq i,j,k \leq n$. The operator $e^{t\Delta} \mathbb{P} (\nabla  \cdot \ \ )$ is a convolution operator $\sum_{j,k=1}^{n}K_{i,j,k}(t)*f_{j,k}$ with
$$
K_{i,j,k}(t,x) := \frac{1}{t^{\frac{n+1}{2}}}k_{i,j,k}\left( \frac{x}{\sqrt{t}} \right)
$$
and for each multi-index $\eta$
$$
k_{i,j,k} \in C^{\infty}(\Rn), \qquad
(1+|x|)^{n+ 1 +|\eta|}\partial^{\eta}k_{i,j,k} \in L^{\infty}(\Rn).
$$

\end{proposition}
We conclude the section with an useful equivalence result:
\begin{theorem}[\cite{Lem}]\label{DiffvsInt}
Let 
\begin{equation}\label{HPEquiv}
u \in \cap_{s<T} \left( L^{2}_{t}L^{2}_{uloc,x}\left( (0,s) \times \Rn \right) \right).
\end{equation}
The following are equivalent:
\begin{enumerate}
\item $(u,P)$ is a weak solution of
\begin{equation}
\left \{
\begin{array}{rcccl}
\partial_{t}u + (u \cdot \nabla) u - \Delta u & = & - \nabla P & \quad \mbox{in}& \quad (0,T)\times \mathbb{R}^{n}  \\
\nabla \cdot u & = & 0 & \quad \mbox{in} & \quad [0,T)\times \mathbb{R}^{n} \\
u & = & u_{0} & \quad \mbox{in}& \quad \{0\} \times \mathbb{R}^{n}, 
\end{array}\right.
\end{equation}
\item $(u,P)$ is a solution of the integral problem
\begin{equation}
\left \{
\begin{array}{rclcl}
 u & = & e^{t\Delta}u_{0} - \int_{0}^{t}e^{(t-s)\Delta}\mathbb{P} \nabla \cdot 
 (u \otimes u) \ ds &\quad \mbox{in} \quad & [0,T)\times \mathbb{R}^{n}  \\
\nabla \cdot u & = & 0 &\quad \mbox{in} \quad & [0,T)\times \mathbb{R}^{n} \\
P & = &  R \otimes R \ (u \otimes u) & \quad \mbox{in} \quad & [0,T)\times \mathbb{R}^{n}. 
\end{array}\right.
\end{equation}
\end{enumerate}
\end{theorem}

\section{Time decay estimates for the heat and Oseen kernels}

We prove time decay estimates for the operators
$e^{t \Delta}$ and $e^{t \Delta } \mathbb{P} (\nabla \cdot \ \ )$.
These estimates turn out to be fundamental in the study of 
the Navier--Stokes problem with small data since 
the pioneering work of Kato \cite{Kat}. Following the 
same philosophy we take advantage of them in order to 
get regularity criteria. This is natural by 
working with the integral formulation (\ref{CauchyNSInt}). 

We investigate the connection between homogeneous weights
and angular regularity by working in $L^{p}_{|x|}L^{\p}_{\theta}$ spaces. 
In particular we show that higher angular 
integrability allows to consider a larger set of weights.

 As mentioned the idea of distinguish radial and angular 
 integrability often occurs in 
 harmonic analysis and PDE's. 
 In particular we refer to \cite{DL} where this technology has been applied to recover in a more general setting the improvements to Sobolev embeddings and Caffarelli-Kohn-Nirenberg inequalities known in the radial case by \cite{DenapoliDrelichmanDuran09-a, DenapoliDrelichman09-aBis, DenapoliDrelichmanDuran10-a, DenapoliDrelichmanDuran11-a, Rubin83-a}.

We need the following

\begin{lemma}[\cite{DL}]\label{cor:nonhom}
  Let $n \geq 2$ and $1\le p\le q\le\infty$, 
  $1\le\p\le\q\le\infty$. Assume
  $\alpha,\beta,\gamma$ satisfy the set of conditions
  \begin{equation}\label{eq:condDL}
    \beta > - \frac nq,\qquad \alpha<\frac{n}{p'},\qquad
    \alpha - \beta \ge (n-1)
      \left(\frac1q-\frac1p+\frac{1}{\p}-\frac{1}{\q}\right),
  \end{equation}
  \begin{equation}\label{eq:condabg}
    \alpha - \beta + \gamma>n\left(1+\frac1q-\frac1p\right).
  \end{equation}
  Then 
  \begin{equation}\label{ourHLS}
    \||x|^{\beta}S_{\gamma} \phi\|
        _{L^{q}_{|x|}L^{\q }_{\theta}} 
    \le C
    \| |x|^{\alpha} \phi\|_{L^{p}_{|x|}L^{\p }_{\theta}},
    \end{equation}
where 
$$
S_{\gamma} \phi := \int_{\Rn} K(x-y)\phi(y) \ dy, 
$$
and the kernel $K$ satisfies 
$$
|K(x)| \leq \frac{Const}{(1+|x|^{2})^{\gamma/2}}.
$$
\end{lemma}

\begin{remark} 
Let point out that:

\begin{itemize} 
   \item The assumptions $\beta > - \frac nq, \alpha<\frac{n}{p'}$ are necessary
   to ensure local integrability;
   \item The assumption (\ref{eq:condabg}) is due to the smoothness of the kernel 
   in the origin. It is less restrictive than its counterpart in the 
   homogeneous case (see \cite{DL}) 
   $$
   \alpha-\beta+\gamma = n\left(1+\frac1q-\frac1p\right),
   $$ 
   that follows by scaling;
   \item The assumption 
   $$
   \alpha-\beta \ge (n-1) \left(\frac1q-\frac1p+\frac{1}{\p}-\frac{1}{\q}\right)
   $$ 
   follows by testing the inequality under 
   translations.       
\end{itemize}  
\end{remark}
It is useful to define the quantity
\begin{equation}\label{LambdaDef}
\Lambda (\alpha, p, \p) := \alpha + \frac{n-1}{p} - \frac{n-1}{\p}.
\end{equation}
Notice that 
$$
\alpha-\beta \ge (n-1) \left(\frac1q-\frac1p+\frac{1}{\p}-\frac{1}{\q} \right) 
\Leftrightarrow
\Lambda(\alpha,p,\p) \ge \Lambda(\beta,q,\q).$$ 
This notation is more convenient for our purposes; we use also simply $\Lambda_{\alpha}$ when the values of $p,\p$ will be clear by the context.

\begin{proposition}\label{PDecayCor}
Let $n\geq 2$, $1 \leq p \leq q \leq +\infty$ and $1 \leq \p \leq \q \leq +\infty$. Assume further that $\alpha, \beta$ satisfy the set of conditions 
 \begin{equation}\label{eq:condDL(Heat)}
    \beta > -\frac nq,\qquad \alpha<\frac{n}{p'}, \qquad
    \Lambda (\alpha,p,\p) \geq \Lambda (\beta,q,\q).
  \end{equation}
Then for each multi-index $\eta$
\begin{enumerate}
         \item \begin{equation}\label{PHeatDer}
     \||x|^{\beta} \partial^{\eta} e^{t\Delta}u_{0}\|_{L^{q}_{|x|}L^{\q}_{\theta}} 
     \leq \frac{c_{\eta}}{t^{(|\eta| + \frac{n}{p}-\frac{n}{q} + \alpha-\beta)/2}} 
     \| |x|^{\alpha} u_{0}\|_{L^{p}_{|x|}L^{\p}_{\theta}}, \qquad t>0,
     \end{equation}
     provided that $|\eta| + \frac{n}{p}-\frac{n}{q} + \alpha-\beta \geq0$,
      \item \begin{equation}\label{PGHeatDer}
       \||x|^{\beta} \partial^{\eta} e^{t\Delta} \mathbb{P} \nabla 
       \cdot F\|_{L^{q}_{|x|}L^{\q}_{\theta}} 
       \leq \frac{d_{\eta}}{t^{(1 + |\eta| + \frac{n}{p} -\frac{n}{q} +\alpha -\beta)/2}} \| |x|^{\alpha} F\|_{L^{p}_{|x|}L^{\p}_{\theta}}, \qquad t>0,
         \end{equation} 
provided that $1+ |\eta| + \frac{n}{p}-\frac{n}{q} + \alpha-\beta > 0$.
\end{enumerate}

\end{proposition}

\begin{proof}
The proof follows by Lemma (\ref{cor:nonhom}) and scaling considerations. At first notice
\begin{equation}\label{HeatScaling}
e^{t\Delta} \phi = S_{\sqrt{t}}e^{\Delta}S_{1/\sqrt{t}} \phi,
\end{equation}
where $S_{\lambda}$ is defined by
\begin{equation}
(S_{\lambda}\phi)(x)= \phi\left( \frac{x}{\lambda} \right).
\end{equation}
Then 
\begin{equation}
\| |x|^{\beta} \partial^{\eta} S_{\lambda} \phi \|_{\Lqqtilde} = \lambda^{\frac{n}{q} +\beta -|\eta|}\||x|^{\beta}\phi\|_{\Lqqtilde}.
\end{equation}
We get
\begin{eqnarray}
\||x|^{\beta} \partial^{\eta} e^{t\Delta} u_{0}\|_{\Lqqtilde} &=& 
\| |x|^{\beta}\partial^{\eta} S_{\sqrt{t}}e^{\Delta}S_{1/ \sqrt{t}} u_{0}\|_{\Lqqtilde} \nonumber \\
&=& t^{(\frac{n}{q} + \beta - |\eta|)/2} \||x|^{\beta} (\partial^{\eta} e^{\Delta})S_{1/\sqrt{t}} u_{0}\|_{\Lqqtilde} \nonumber \\
&\leq & \frac{c_{\eta}}{t^{(-\frac{n}{q} -\beta + |\eta |)/2}}\||x|^{\alpha}S_{1/ \sqrt{t}}u_{0}\|_{\Lpptilde} \nonumber \\
&=& \frac{c_{\eta}}{t^{(|\eta| + \frac{n}{p}-\frac{n}{q}+\alpha -\beta)/2}}\||x|^{\alpha} u_{0}\|_{\Lpptilde}, \nonumber
\end{eqnarray}
provided that
$$ 
\Lambda(\alpha,p,\p) \geq \Lambda(\beta,q,\q).
$$
Notice that the third condition in (\ref{eq:condDL}) is 
trivially satisfied by the heat kernel.
To prove (\ref{PGHeatDer}) we have to work with the operator 
$e^{t\Delta} \mathbb{P} (\nabla \cdot \ \ ) $
that is (see Lemma \ref{OseenDecayFinale}) 
a convolution operator with a kernel $K$ such that
\begin{equation}\label{OseenInProofScaling}
K_{j,k,m}(t,x) := k_{j,k,m} \left( \frac{x}{\sqrt{t}} \right),
\end{equation}
and 
\begin{equation}\label{OseenInProof}
(1+|x|)^{1 + n + |\mu |}\partial^{ \mu} k_{j,k,m} \in L^{\infty}(\Rn),
\end{equation}
for each multi-index $\mu$. By (\ref{OseenInProofScaling}) follows
\begin{equation}\label{OssenScaling}
K(t) * \phi =\frac{1}{\sqrt{t}} S_{\sqrt{t}} k * S_{1/\sqrt{t}} \phi.
\end{equation}
So 
\begin{eqnarray}
\||x|^{\beta} \partial^{\eta} e^{t\Delta} \mathbb{P} \nabla \cdot F\|_{\Lqqtilde} &=&  
\| |x|^{\beta}  \partial^{\eta}K(t) * F\|_{\Lqqtilde}
\nonumber \\
&=& \frac{1}{\sqrt{t}}
\| |x|^{\beta}\partial^{\eta} S_{\sqrt{t}} k * S_{1/ \sqrt{t}} F\|_{\Lqqtilde}                     \nonumber \\
&=& \frac{1}{\sqrt{t}}t^{(\frac{n}{q} + \beta  - |\eta|)/2} \||x|^{\beta} (\partial^{\eta} k) * S_{1/ \sqrt{t}} F\|_{\Lqqtilde} \nonumber \\
&\leq & \frac{d_{\eta}}{t^{(-\frac{n}{q} -\beta +1 +  |\eta |)/2}}\||x|^{\alpha}S_{1/ \sqrt{t}} F\|_{\Lpptilde} \nonumber \\
&=& \frac{d_{\eta}}{t^{(1+ |\eta| + \frac{n}{p}-\frac{n}{q}+\alpha -\beta)/2}}\||x|^{\alpha} F\|_{\Lpptilde}, \nonumber
\end{eqnarray}
provided that $\Lambda_{\alpha} \geq \Lambda_{\beta}$. Notice that the optimal choice of $\gamma$ allowed by (\ref{OseenInProof}) is $\gamma = 1 + n + |\eta|$ that leads to
$$\alpha - \beta + 1 + n + |\eta| > n \Big( 1+\frac{1}{q}- \frac{1}{p} \Big) \Rightarrow 
1+ |\eta| + \frac{n}{p}-\frac{n}{q}+\alpha -\beta > 0.$$
\end{proof}

It's remarkable that the restriction $\Lambda_{\alpha} \geq \Lambda_{\beta}$ 
can be removed by localizing the estimate in the interior of
a space-time parabola above the origin. The size of the parabola
depends on the values of the difference 
$\Lambda_{\alpha} - \Lambda_{\beta}$ and increases 
as $\Lambda_{\alpha}- \Lambda_{\beta} \rightarrow 0^{-}$. 
In the limit case $\Lambda_{\alpha}=\Lambda_{\beta}$ we 
recover in fact Proposition \ref{PDecayCor}.
\begin{proposition}\label{LocalPDecay}
Let $n\geq 2$, $1 \leq p \leq q \leq +\infty$ and $1 \leq \p \leq \q \leq +\infty$. Assume further that $\alpha, \beta$ satisfy the set of conditions 
 \begin{equation}\label{eq:condDL(Heat)Loc}
    \beta > -\frac nq,\qquad \alpha<\frac{n}{p'}, \qquad
    \Lambda (\alpha,p,\p) < \Lambda (\beta,q,\q),
    \end{equation}
and define
$$
\Lambda_{\alpha,\beta} := \Lambda (\alpha,p,\p) - \Lambda (\beta,q,\q).
$$    
Let then 
$$
\Pi(R) := \left\{ (t,x) \in \Rpiu \times \Rn : 
\quad
\frac{|x|}{\sqrt{t}} \leq R \right\},
$$   
for each muti index $\eta$
\begin{enumerate}
         \item \begin{equation}\label{PHeatDerLoc}
     \| \mathbbm{1}_{\Pi(R)} |x|^{\beta} \partial^{\eta} e^{t\Delta}u_{0}\|_{L^{q}_{|x|}L^{\q}_{\theta}} 
     \leq \frac{c_{\eta}R^{-\Lambda_{\alpha,\beta}}}{t^{(|\eta| + \frac{n}{p}-\frac{n}{q} + \alpha-\beta)/2}} 
     \| |x|^{\alpha} u_{0}\|_{L^{p}_{|x|}L^{\p}_{\theta}}, \qquad t>0,
     \end{equation}
     provided that 
     $|\eta| + \frac{n}{p}-\frac{n}{q} + \alpha-\beta \geq0,
     \quad \Lambda_{\alpha, \beta} < 0,     
     $
      \item \begin{equation}\label{PGHeatDerLoc}
       \| \mathbbm{1}_{\Pi(R)} |x|^{\beta} \partial^{\eta} e^{t\Delta} \mathbb{P} \nabla \cdot F\|_{L^{q}_{|x|}L^{\q}_{\theta}} 
       \leq \frac{d_{\eta}R^{-\Lambda_{\alpha,\beta}}}{t^{(1 + |\eta| + \frac{n}{p} -\frac{n}{q} +\alpha -\beta)/2}} \| |x|^{\alpha} F\|_{L^{p}_{|x|}L^{\p}_{\theta}}, \qquad t>0,
         \end{equation} 
provided that 
$
1+ |\eta| + \frac{n}{p}-\frac{n}{q} + \alpha-\beta > 0,
\quad \Lambda_{\alpha, \beta} < 0.
$
\end{enumerate}

\end{proposition}

\begin{proof}
Let us write simply $\Lambda$ instead of $\Lambda_{\alpha,\beta}$. Of course
$$
\Lambda < 0 \quad \Rightarrow \quad  R^{-\Lambda} \Big| \frac{x}{\sqrt{t}} \Big|^{\Lambda} \geq 1, \quad
\mbox{if} \quad (t,x)\in \Pi(R).
$$
Then 
\begin{eqnarray}
\|\mathbbm{1}_{\Pi(R)} |x|^{\beta} \partial^{\eta} e^{t\Delta} u_{0}\|_{\Lqqtilde} &=& 
\|\mathbbm{1}_{\Pi(R)} |x|^{\beta}\partial^{\eta} S_{\sqrt{t}}e^{\Delta}S_{1/ \sqrt{t}} u_{0}\|_{\Lqqtilde} \nonumber \\
&\leq & \frac{R^{-\Lambda}}{t^{\Lambda/2}}\| |x|^{\beta +\Lambda}\partial^{\eta} S_{\sqrt{t}}e^{\Delta}S_{1/ \sqrt{t}} u_{0}\|_{\Lqqtilde} \nonumber \\
&=& \frac{R^{-\Lambda}}{t^{\Lambda /2}} t^{(\frac{n}{q} + \beta + \Lambda - |\eta|)/2} \||x|^{\beta +\Lambda} (\partial^{\eta}e^{\Delta}) S_{1/\sqrt{t}} u_{0}\|_{\Lqqtilde} \nonumber \\
&\leq & \frac{c_{\eta}}{t^{(-\frac{n}{q} -\beta + |\eta |)/2}}\||x|^{\alpha}S_{1/ \sqrt{t}}u_{0}\|_{\Lpptilde} \nonumber \\
&=& \frac{c_{\eta}}{t^{(|\eta| + \frac{n}{p}-\frac{n}{q}+\alpha -\beta)/2}}\||x|^{\alpha} u_{0}\|_{\Lpptilde}, \nonumber
\end{eqnarray}
where the indexes relationships are consistent because
$$
\Lambda_{\alpha} \geq \Lambda(\Lambda_{\alpha,\beta} +\beta,p,\p) = \Lambda(\Lambda_{\alpha}-\Lambda_{\beta} + \beta,p,\p)= \Lambda_{\alpha}.
$$
The proof of (\ref{PGHeatDerLoc}) is analogous. 
\end{proof}
\begin{remark}
We have observed observed that the inequalities 
hold with an additional factor $R^{-\Lambda}$ after 
localization in the interior of a space-time parabola. 
Notice that this factor goes to $1$ as 
$\Lambda \rightarrow 0^{-}$. To get a constant independent on
$\Lambda$ it is instead necessary to restrict the 
size of the parabola. If we chose the constant equal to $K$, we need to restrict to
$$
\Pi(K) := \left\{  \frac{|x|}{\sqrt{t}}  \leq K^{-\frac{1}{\Lambda}} \right\}.  
$$ 
Notice that $\Pi (K)$ fills the whole space-time as $\Lambda \rightarrow 0^{-}$.
\end{remark}

Then integral estimates can be obtained
by the time decay properties.
Let us introduce another useful notation 
\begin{equation}\label{OmegaDef}
\Omega (\alpha,p,s) := \alpha + \frac np + \frac 2s.
\end{equation}

\begin{proposition}\label{IDecayCor}
Let $n\geq 2$, $1 \leq p \leq q < \frac{np}{(|\eta| + \alpha-\beta)p + n-2}$, 
$p < r < +\infty$ and $ 1 \leq \p \leq \q \leq +\infty$. 
Assume further that $\alpha, \beta$ satisfy 
 \begin{equation}\label{eq:condDL(IHeat)}
    \beta > -\frac nq,\qquad \alpha<\frac{n}{p'}.
   \end{equation}
Then for each multi-index $\eta$
\begin{equation}\label{IHeat}
     \| |x|^{\beta} \partial^{\eta} e^{t\Delta}u_{0}\|_{L^{r}_{t}L^{q}_{|x|}L^{\q}_{\theta}} 
     \leq c_{\eta} 
     \| |x|^{\alpha} u_{0}\|_{L^{p}_{|x|}L^{\p}_{\theta}}, \qquad t>0,
     \end{equation}
provided that
     \begin{equation}\label{OmegaScaling}
  |\eta| + \Omega(\alpha, p, \infty) = \Omega (\beta,q,r), 
  \quad \Lambda (\alpha,p,\p) \geq \Lambda (\beta,q,\q);
  \end{equation}
and 
\begin{equation}\label{IHeatLoc}
     \| \mathbbm{1}_{\Pi(R)} |x|^{\beta} \partial^{\eta} e^{t\Delta}u_{0}\|_{L^{r}_{t}L^{q}_{|x|}L^{\q}_{\theta}} 
     \leq c_{\eta} R^{-\Lambda_{\alpha,\beta}} 
     \| |x|^{\alpha} u_{0}\|_{L^{p}_{|x|}L^{\p}_{\theta}}, \qquad t>0,
     \end{equation}
provided that
     \begin{equation}\label{OmegaScalingLoc}
  |\eta| + \Omega(\alpha, p, \infty) = \Omega (\beta,q,r), 
  \quad \Lambda_{\alpha,\beta} := \Lambda (\alpha,p,\p) - \Lambda (\beta,q,\q) < 0,  
  \end{equation}
  \end{proposition}

\begin{proof}
By the time decay
 \item \begin{equation*}
     \||x|^{\beta} \partial^{\eta} e^{t\Delta}u_{0}\|_{L^{q}_{|x|}L^{\q}_{\theta}} 
     \leq \frac{c_{\eta}}{t^{(|\eta| + \frac{n}{p}-\frac{n}{q} + \alpha-\beta)/2}} 
     \| |x|^{\alpha} u_{0}\|_{L^{p}_{|x|}L^{\p}_{\theta}},
     \end{equation*}
follows that $\partial^{\eta}e^{t\Delta}u_{0}$ is bounded in the Lorentz space $L^{r,\infty}(\mathbb{R}^{+}; L^{q}_{|x|^{\beta q}d|x|}L^{\q}_{\theta})$ provided that $|\eta| + \Omega(\alpha, p, \infty) = \Omega (\beta,q,r)$. In fact
\begin{eqnarray} 
\|\||x|^{\beta}\partial^{\eta}e^{t\Delta}u_{0}\|_{\Lqqtilde}\|_{L^{r,\infty}_{t}} &\leq &
   c_{\eta}  \left\| \frac{1}{t^{(|\eta| + \frac{n}{p}-\frac{n}{q} + \alpha-\beta)/2}} 
     \| |x|^{\alpha} u_{0}\|_{L^{p}_{|x|}L^{\p}_{\theta}}\right\|_{L^{r,\infty}_{t}} \nonumber \\  
&\leq & c_{\eta} \left\|  \frac{1}{t^{(|\eta| + \frac{n}{p}-\frac{n}{q} + \alpha-\beta)/2}} \right\|_{L^{r,\infty}}
        \|u_{0}\|_{\Lpptilde}  \nonumber \\
        &\leq & c_{\eta} \|u_{0}\|_{\Lpptilde}, \nonumber
\end{eqnarray}     
when
 $$
 (|\eta| + \frac{n}{p} -\frac{n}{q} +\alpha -\beta)/2 = \frac{1}{r} \Rightarrow |\eta| + \Omega(\alpha, p, \infty) = \Omega (\beta,q,r).
 $$
Let now consider $(\alpha_{0},\beta_{0}, p_{0},\p_{0},q_{0},\q_{0},r_{0})$, $(\alpha_{1},\beta_{1}, p_{1},\p_{1},q_{1},\q_{1},r_{1})$ such that the assumptions of the Theorem are satisfied. We have the bounded operators
\begin{equation} 
\partial^{\eta} e^{t\Delta}: 
\left \{    
    \begin{array}{ccc}
     \Lpptildealphazero & \longrightarrow & L^{r_{0},\infty}_{t} \Lqqtildebetazero \\
     & & \\
     & & \\
     \Lpptildealphauno & \longrightarrow & L^{r_{1},\infty}_{t} \Lqqtildebetauno .  
      \end{array}
      \right.
\end{equation}
and we can use real interpolation with 
parameters $(\xi, r_{\xi}), 0 \leq \xi \leq 1$ provided that
\begin{equation}\label{rConstr}
p_{\xi} < r_{\xi},
\end{equation}
$$
\frac{1}{p_{\xi}} = (1- \xi)\frac{1}{p_{0}} + \frac{\xi}{ p_{1}},
$$  
$$
\frac{1}{q_{\xi}} = (1- \xi)\frac{1}{q_{0}} + \frac{\xi}{ q_{1}},
$$  
$$
\frac{1}{r_{\xi}} = (1- \xi)\frac{1}{r_{0}} + \frac{\xi}{ r_{1}},
$$  
$$
\frac{1}{\p_{\xi}} = (1- \xi)\frac{1}{\p_{0}} + \frac{\xi}{ \p_{1}},
$$  
$$
\frac{1}{\q_{\xi}} = (1- \xi)\frac{1}{\q_{0}} + \frac{\xi}{ \q_{1}},
$$  
$$
\alpha_{\xi}= (1-\xi)\alpha_{0} +\xi \alpha_{1},
$$
$$
\beta_{\xi}= (1-\xi)\beta_{0} +\xi \beta_{1}.
$$
Then since
$$
  \Big(L^{r_{0},\infty}_{t} \Lqqtildebetazero, L^{r_{1},\infty}_{t} \Lqqtildebetauno \Big)_{\xi,r_{\xi}} = L^{r_{\xi}}_{t} \Lqqtildebetaxi,
$$ 
we get the bounded operators
$$
\partial^{\eta}e^{t\Delta} u_{0} :
\Lpptildealphaxi \rightarrow   L^{r_{\xi}}_{t} \Lqqtildebetaxi. 
$$
It is now straightforward to check that the  indexes satisfy (\ref{eq:condDL(IHeat)}, \ref{OmegaScaling}) and the other assumptions. 
In particular (\ref{rConstr}) is ensured by $q_{\xi} < \frac{np_{\xi}}{(|\eta| + \alpha_{\xi} - \beta_{\xi}) p_{\xi} + n-2}$. 

Of course this method misses the endpoint $r=1$.
The estimates (\ref{IHeatLoc}) can be proved in the same 
way by using the localized time decay. 

\end{proof}
Then we bound the Duhamel term:
\begin{proposition}\label{DDecayCor}
Let $n\geq 2$, $2 \leq p \leq 2q \leq +\infty$, $2 < s < 2r < +\infty $ and $2 \leq \p \leq 2\q \leq +\infty$. Assume further that $\alpha, \beta$ satisfy 
 \begin{equation}\label{eq:condDL(DHeat)}
    \beta > -\frac nq,\qquad \alpha < \frac{n}{2} - \frac{n}{p},
   \end{equation}
then for each multi-index $\eta$
  \begin{equation}\label{DGHeat}
       \left\||x|^{\beta} \partial^{\eta} \int_{0}^{t} e^{(t-s)\Delta} 
       \mathbbm{P} \nabla  \cdot 
       (u \otimes u) \ ds \right\|_{L^{r}_{t}L^{q}_{|x|}L^{\q}_{\theta}} 
       \leq  d_{\eta} \| |x|^{\alpha} 
       u \|^{2}_{L^{s}_{t}L^{p}_{|x|}L^{\p}_{\theta}}, \qquad t>0, 
         \end{equation} 
provided that
       \begin{equation}\label{DGOmegaScaling}
  2 \Omega(\alpha, p, s) = \Omega (\beta,q,r) + 1 - |\eta |, \qquad
  2 \Lambda (\alpha,p,\p) \geq \Lambda (\beta,q,\q);
       \end{equation}
in particular (for $2 < r < \infty$)
   \begin{equation}\label{DGHeatq=p}
       \left\||x|^{\beta}  \int_{0}^{t} e^{(t-s)\Delta} \mathbbm{P} \nabla 
       \cdot (u \otimes u) \  ds\right\|_{L^{r}_{t}L^{q}_{|x|}L^{\q}_{\theta}} 
       \leq  d_{\eta} \| |x|^{\beta} 
       u \|^{2}_{L^{r}_{t}L^{q}_{|x|}L^{\q}_{\theta}}, \qquad t>0, 
         \end{equation} 
provided that
         \begin{equation}\label{DGOmegaScalingq=p}
          \frac 2r + \frac nq   = 1 - \beta , 
         \qquad \Lambda(\beta,q,\q) \geq 0.
         \end{equation}
     \end{proposition}

\begin{proof}
By Minkowski inequality and (\ref{PHeatDer})
\begin{eqnarray}
  & &\left\||x|^{\beta} \partial^{\eta} \int_{0}^{t} e^{(t-s)\Delta}\mathbb{P} 
   \nabla \cdot F(x,s) \ ds \right\|_{\Lrqqtilde}   \nonumber \\ 
  & \leq &\left\| \int_{\Rpiu} \| |x|^{\beta} \partial^{\eta} e^{(t-s)\Delta}\mathbb{P} 
   \nabla \cdot F \|_{\Lqqtilde} \ ds\right\|_{L^{r}_{t}}  \nonumber \\
  &\leq & d_{\eta} \left\| \int_{\mathbb{\Rpiu}} 
   \frac{1}{(t-s)^{(1 + |\eta | 
   +\frac{n}{p_{0}} -\frac{n}{q} +\alpha_{0} -\beta)/2}}\||x|^{\alpha_0}F \|_{\L^{p_0}_{|x|}L^{\p_0}_{\theta}}   
   \ ds  \right\|_{L^{r}_{t}},   \nonumber 
\end{eqnarray}
provided that
\begin{equation}\label{DuamhCondRemark}
\p_0 \leq \q, \quad p_0 \leq q \quad 1 + |\eta | + \frac{n}{p_{0}} - \frac{n}{q} + \alpha_{0} - \beta > 0, \quad \Lambda_{\alpha_0} \geq \Lambda_{\beta}.
\end{equation}
Let then
\begin{equation}\label{WYScalingInCor}
1+ \frac{1}{r} = \frac{1}{s_{0}} + \frac{1}{k},
\end{equation}
and use the Young inequality in Lorentz spaces
$$
\| \cdot \|_{L^{r}} \leq \| \cdot \|_{L^{s_{0}}} \| \cdot \|_{L^{k,\infty}},
$$
that is allowed if $1<r,s_{0},k < +\infty$\footnote{We are not interested in the case $k = +\infty$ to which corresponds a singular behaviour.}. We get
\begin{eqnarray}
  & &\left\||x|^{\beta} \partial^{\eta} \int_{0}^{t} e^{(t-s)\Delta}\mathbb{P} 
   \nabla \cdot F(x,s) \ ds \right\|_{\Lrqqtilde}   \nonumber \\ 
   &\leq & d_{\eta} \||x|^{\alpha_0} F\|_{L^{s_{0}}_{t}L^{p_{0}}_{|x|}L^{\p_{0}}_{\theta}} 
   \left\|\int_{\mathbb{\Rpiu}} \frac{dt}{t^{(1+|\eta | +\frac{n}{p_{0}} 
   -\frac{n}{q} +\alpha_{0}-\beta)/2}}\right\|_{L^{k,\infty}_{t}}   \nonumber \\
   & \leq &\ d_{\eta} \||x|^{\alpha_0} 
 F\|_{L^{s_{0}}_{t}L^{p_{0}}_{|x|}L^{\p_{0}}_{\theta}}, \nonumber 
\end{eqnarray}
provided that
\begin{equation}\label{DuhamProvvScal}
p_0 \leq q \quad (1 + |\eta | + \frac{n}{p_{0}} - \frac{n}{q} + \alpha_{0} - \beta)/2=\frac{1}{k}, \quad \Lambda_{\alpha_0} \geq \Lambda_{\beta},
\end{equation}
since
$$
\left\| \int_{\Rpiu} \frac{dt}{t^{1/k}}\right\|_{L^{k,\infty}_{t}} =1.
$$
By (\ref{WYScalingInCor}) and the second in (\ref{DuhamProvvScal})
\begin{equation}\label{DuhamProvvScal2}
\Omega(\alpha_{0}, p_{0}, s_{0}) = 1-|\eta | +\Omega(\beta,q,r).
\end{equation}
We now specify $F=u \otimes u$
\begin{eqnarray}
  & &\left\||x|^{\beta} \partial^{\eta} \int_{0}^{t} e^{(t-s)\Delta}\mathbb{P} 
   \nabla \cdot (u\otimes u)(x,s)  \ ds \right\|_{\Lrqqtilde}   \nonumber \\ 
  &\leq &  c_{\eta} \||x|^{\alpha_{0}} 
   |u|^{2}\|_{L^{s_{0}}_{t}L^{p_{0}}_{|x|}L^{\p_{0}}_{\theta}}  \nonumber \\
 &\leq &  c_{\eta} \||x|^{\alpha_{0}/2} 
   |u|\|^{2}_{L^{2s_{0}}_{t}L^{2p_{0}}_{|x|}L^{2\p_{0}}_{\theta}} \\
 & \leq  &c_{\eta} \||x|^{\alpha}u\|^{2}_{\Lspptilde}, \nonumber
  \end{eqnarray}
where we have set 
\begin{equation}\label{DuhamPivot}
(\alpha_{0}/2, 2s_{0},2p_{0},2\p_{0})=(\alpha,s,p,\p).
\end{equation}
Notice that $2\Omega(\alpha,s,p)= \Omega(\alpha_{0},s_{0},p_{0})$, $2\Lambda_{\alpha}=\Lambda_{\alpha_{0}}$ so (\ref{DuhamProvvScal2}), (\ref{DuhamPivot}) and the last in (\ref{DuhamProvvScal}) lead to 
$$
2\Omega(\alpha,p,s) = \Omega(\beta,q,r) +1 - |\eta |, \qquad
2\Lambda_{\alpha} \geq \Lambda_{\beta}.
$$
Finally notice that (\ref{WYScalingInCor}) and (\ref{DuhamProvvScal}) imply 
$$
r > s_{0}=s/2, \qquad q \geq p_{0}=p/2, \qquad \q \geq \p_{0}=\p/2.
$$
These conditions are furthermore consistent with the choice
$(\alpha,s,p,\p)=(\beta,r,q,\q)$, 
in such a way we recover inequality (\ref{DGHeatq=p})
$$
\left\||x|^{\beta} \partial^{\eta} \int_{0}^{t} e^{(t-s)\Delta} \mathbbm{P} \nabla 
       \cdot (u \otimes u) \  ds \right\|_{L^{r}_{t}L^{q}_{|x|}L^{\q}_{\theta}} 
       \leq  d_{\eta} \| |x|^{\beta} 
       u \|^{2}_{L^{r}_{t}L^{q}_{|x|}L^{\q}_{\theta}},
$$
provided that
$$
\Omega(\beta,q,r) = 1-|\eta|, \qquad \Lambda(\beta,q,\q) \geq 0.
$$

\end{proof}

\section{Proof of the main results}

We refer to the relations 
$$
\frac{2}{s} + \frac{n}{p} = 1 - \alpha, 
\quad \alpha_{0} = 1- \frac{n}{p_{0}}, 
\quad
\frac{2}{s} + \frac{n}{p} = 1 - \beta,
$$
as scaling assumptions.

As mentioned Theorems \ref{OurYZTheorem}, \ref{OurYZTheoremLoc} actually hold under weaker assumptions on $u_{0}$, we prove in fact:

\begin{theorem}
Theorem \ref{OurYZTheorem} holds if 
$u_{0} \in L^{2}(\Rn)$ is a divergence free vector field and 
$$
\| |x|^{\alpha_{0}} u_{0}\|_{L^{p_{0}}_{|x|}L^{\p_{0}}_{\theta}} < +\infty
$$ 
with
\begin{equation}\label{OurYZCond0}
\alpha_{0} \in [(2-n)/2, 2/(2+n)), \quad 
\alpha_{0} = 1- \frac{n}{p_{0}}, \quad 
\p_{0} \leq \frac{\p_{G}}{2},
\end{equation}
\begin{equation}\label{OurYZCond0Complicata}
\left \{
\begin{array}{lcr}
2 \leq p_{0} \leq \p_{G}/2   &  \mbox{if} & \p_{G} \leq 2n   \\
2 \leq p_{0} \leq \p_{G}/2, \quad p_{0} < \frac{2\p_{G}}{\p_{G} - 2n}    &  \mbox{if} & \p_{G} > 2n;
\end{array}\right.
\end{equation}
or
\begin{equation}\label{OurYZCond0bis}
\alpha_{0} \in [(2-n)/2, 2/(2+n)),  \quad
\alpha_{0} = 1- \frac{n}{p_{0}}, \quad 
\p_{0} \leq \frac{p}{2},
\end{equation}
\begin{equation}\label{OurYZCond0ComplicataBis}
\left \{
\begin{array}{lcr}
2 \leq p_{0} \leq p/2   &  \mbox{if} & p \leq 2n   \\
2 \leq p_{0} \leq p/2, \quad p_{0} < \frac{2 p}{p - 2n}    &  \mbox{if} & p > 2n; 
\end{array}\right.
\end{equation}
while $u$ has to satisfy (\ref{HPEquiv}) and 
(\ref{OurYZCondition0}, \ref{OurYZBound1},
\ref{OurYZCondition1}, \ref{OurYZCondition2}, \ref{OurYZCondition3}),
or
(\ref{OurYZCondition0bis}, \ref{OurYZBound1bis},
\ref{OurYZCondition1bis}, \ref{OurYZCondition2bis}, \ref{OurYZCondition3bis}).
\end{theorem}

\begin{proof}
Since we want to use the regularity condition (\ref{SerrinNorm}) we need to show that
\begin{equation}\label{OurYZCondition1Proof}
\|u\|_{L^{r}_{T}L^{q}_{x}} < +\infty, \quad \mbox{with} \quad \frac{2}{r} +\frac{n}{q} =1.
\end{equation}
Let's start by the integral representation
\begin{equation}\nonumber
u  =  e^{t \Delta}u_{0} - \int_{0}^{t}e^{(t-s)\Delta}\mathbb{P}\nabla \cdot 
(u \otimes u)(s) \ ds 
\end{equation}
and distinguish the cases $\alpha \in ((1-n)/2,0)$ and $\alpha \in [0,1/2]$.
\ni \subsection*{Case $\alpha \in ((1-n)/2,0)$} 
\begin{eqnarray}\nonumber
\| u \|_{L^{r}_{T}L^{q}_{x}} &\leq & \|e^{t\Delta}u_{0}\|_{L^{r}_{T}L^{q}_{x}}+
\left\|\int_{0}^{t}e^{(t-s)\Delta}\mathbb{P}\nabla \cdot (u \otimes u)(s) \ ds \right\|_{L^{r}_{T}L^{q}_{x}} \\ \nonumber
&=& I + II.
\end{eqnarray}
By the scaling assumption and Proposition \ref{IDecayCor}
\begin{equation}\label{I}
I \leq c_{0} \||x|^{\alpha_{0}}u_{0}\|_{L^{p_{0}}_{|x|}L^{\p_{0}}_{\theta}}, 
\end{equation}
provided that 
\begin{equation}\label{I2}
  p_{0} \leq q < \frac{np_{0}}{p_{0} - 2}, \quad \p_{0} \leq q,
   \quad \Lambda(\alpha_{0}, p_{0}, \p_{0}) \geq 0.
\end{equation}
Actually the condition $\Lambda(\alpha_{0}, p_{0}, \p_{0}) \geq 0$ is not necessary in order to prove the Theorem, we assume it for now in order to avoid some technicalities in the proof. We will show how to remove it at the end of the proof.
We use Proposition \ref{DDecayCor} and scaling
to bound
\begin{equation}\nonumber
II \leq  d_{0} 
\||x|^{\alpha}u\|^{2}_{L^{s}_{T}L^{p}_{|x|}L^{\p_{G}}_{\theta}}
\lesssim
\||x|^{\alpha}u\|^{2}_{L^{s}_{T}\Lpptilde}, 
\end{equation}
provided that
\begin{equation}\label{OurYZCond2Proof}
\Lambda(\alpha,p, \p) \geq 0,
\end{equation}

\begin{equation}\label{OurYZCond3Proof}
2 \leq p \leq +\infty, 
\quad 2 < s < +\infty, 
\quad p/2, \ \p_{G} /2 \leq q, 
\quad s/2 < r.
\end{equation}
Condition (\ref{OurYZCond2Proof}) is ensured by
\begin{equation}\label{alsoThat3}
\p \geq \frac{(n-1)p}{\alpha p +n -1}.
\end{equation}
Notice also that (\ref{alsoThat3}), the scaling and $\alpha <0$ imply $\frac{n}{1-\alpha} < p \leq \frac{1-n}{\alpha}$, so the widest range for $p$ is attained as $\alpha \rightarrow 0^{-}$.
Then we need a couple $(r,q)$ such that (\ref{OurYZCond3Proof}) is consistent with $\frac{2}{r} + \frac{n}{q}=1$. We choose $q = \p_{G}/2 = \max \left( n, \frac{(n-1)p}{2\alpha p + 2(n-1)} \right)$. 
This is allowed by $(1-n)/2 < \alpha$, 
we have indeed
$$
\frac{2}{r} = 1- \frac{n}{q} = 1- \frac{2n \alpha}{n-1} + \frac{2n}{p} \Rightarrow 
\frac{2}{r} - \frac{4}{s} = \frac{1-n -2 \alpha}{n-1},
$$
so
$$
(1-n)/2 < \alpha \Rightarrow s/2 < r; 
$$
Finally (\ref{I2}) becomes 
$$
p_{0} \leq \frac{\p_{G}}{2} < \frac{np_{0}}{p_{0}-2},
$$
that by a straightforward calculation leads to (\ref{OurYZCond0Complicata}) and $\alpha_{0} \in [(2-n)/2, 2/(2+n))$.

\subsection*{Case $\alpha \in [0,1/2)$} 
The only difference is in the choice of $(r,q)$. Here we set $q = p/2$. In such a way (\ref{OurYZCond3Proof}) is ensured by $\alpha < 1/2$, in fact
$$
\frac{2}{r} = 1- \frac{2n}{p} \Rightarrow \frac{2}{r} - \frac{4}{s} = -1 + 2 \alpha,
$$  
so
$$
\alpha < 1/2 \Rightarrow s/2 < r.
$$
Notice that in this case we do not have the restriction $p\leq \frac{1-n}{\alpha}$.
Then (\ref{I2}) becomes
$$
p_{0} \leq \frac{q}{2} < \frac{np_{0}}{p_{0}-2},
$$
that by a straightforward calculation leads to (\ref{OurYZCond0ComplicataBis}), $\alpha_{0} \in [(2-n)/2, 2/(2+n))$.
The choice $q = p/2, \beta = 0$ and the scaling assumptions
force to be $p > 2n$.

We show how the assumption $\Lambda(\alpha_{0}, p_{0}, \p_{0}) \geq 0$ can be removed. 
Let us write simply $\Lambda$ instead of $\Lambda(\alpha_{0}, p_{0}, \p_{0})$ and suppose $\Lambda < 0 $.  We can use the localized estimate (\ref{IHeatLoc}) to get the bound
$$
\|\mathbbm{1}_{\Pi(R)} u \|_{L^{r}_{T}L^{q}_{x}} \leq 
R^{-\Lambda} c_{0} \||x|^{\alpha_{0}} u_{0} \|_{L^{p_{0}}_{|x|}L^{\p_{0}}_{\theta}} +
d_{0} \||x|^{\alpha} u \|_{L^{s}_{T}L^{p}_{|x|}L^{\p}_{\theta}}
$$
where 
$$
\Pi(R) := \left\{ (t,x) \in \Rpiu \times \Rn: \quad
 \frac{|x|}{\sqrt{t}} \leq R \right\}.
$$
So $(0,T) \times \Rn$ is a regular set by taking the limit $R \rightarrow +\infty$.

\end{proof}

\begin{theorem}
Theorem \ref{OurYZTheoremLoc} holds if $u_{0} \in H^{2} \cap L^{2}_{|x|^{2-n}dx}$
is a divergence free vector field such that
$$  
   \| |x|^{\alpha_{0}} u_{0}\|_{L^{p_{0}}_{|x|}L^{\p_{0}}_{\theta}} < +\infty,
   \quad \Lambda(\alpha_{0}, p_{0}, \p_{0}) \geq 0,
$$ 
with
\begin{equation}\label{OurYZCond0Loc}
\alpha_{0} \in \left[ 1-n, \frac{2-n}{2+n}\right), \quad
\alpha_{0} = 1- \frac{n}{p_{0}}, \quad 
\p_{0} \leq \frac{p}{2},
\end{equation}

\begin{equation}\label{OurYZCond0ComplicataLoc}
\left \{
\begin{array}{lcr}
1 \leq p_{0} \leq p/2   &  \mbox{if} & p \leq n   \\
1 \leq p_{0} \leq p/2, \quad p_{0} < \frac{p}{p - n}    &  \mbox{if} & p > n; 
\end{array}\right.
\end{equation}
or
\begin{equation}\label{OurYZCond0bisLoc}
\alpha_{0} \in \left[ 1-(1-\alpha)n, 1- (1-\alpha)\frac{2n}{2+n} \right), \quad
\alpha_{0} = 1- \frac{n}{p_{0}}, \quad 
\p_{0} \leq \frac{p}{2}, 
\end{equation}
\begin{equation}\label{OurYZCond0ComplicataBisLoc}
\frac{1}{1-\alpha} \leq p_{0} \leq \frac{p}{2}, \quad p_{0} < \frac{p}{(1-\alpha)p -n};
\end{equation}
while $u$ has to satisfy (\ref{HPEquiv}) and (\ref{OurYZCondition0Loc}, \ref{OurYZBound1Loc},
\ref{OurYZCondition1Loc}, \ref{OurYZCondition2Loc}, \ref{OurYZCondition3Loc}),
or (\ref{OurYZCondition0bisLoc}, \ref{OurYZBound1bisLoc},
\ref{OurYZCondition1bisLoc}, \ref{OurYZCondition2bisLoc}, \ref{OurYZCondition3bisLoc}).

\end{theorem}

\begin{proof}
Since we want to use directly Theorem \ref{YZTheoremSS} so we need to show that
\begin{equation}\label{OurYZCondition1ProofLoc}
\| |x|^{\beta} u \|_{L^{r}_{T}L^{q}_{x}} < +\infty, \quad \mbox{with} \quad \frac{2}{r} +\frac{n}{q} =1 - \beta.
\end{equation}
Let's start by the integral representation
\begin{equation}\nonumber
u  =  e^{t \Delta}u_{0} - \int_{0}^{t}e^{(t-s)\Delta}\mathbb{P}\nabla \cdot 
(u \otimes u)(s) \ ds.  
\end{equation}
and distinguish the cases $\alpha \in [-1/2,0)$ and $\alpha \in [0,1)$.
\ni \subsection*{Case $\alpha \in [-1/2,0)$} 
\begin{eqnarray}\nonumber
\| |x|^{\beta} u \|_{L^{r}_{T}L^{q}_{x}} &\leq & \| |x|^{\beta} e^{t\Delta}u_{0}\|_{L^{r}_{T}L^{q}_{x}} \\ \nonumber
&+& \left\| |x|^{\beta} \int_{0}^{t}e^{(t-s)\Delta}\mathbb{P}\nabla \cdot (u
\otimes u)(s) \ ds \right\|_{L^{r}_{T}L^{q}_{x}} \\ \nonumber
&=& I + II.
\end{eqnarray}
By the scaling assumption and Poposition \ref{IDecayCor} 
\begin{equation}\label{ILoc}
I \leq c_{0} \||x|^{\alpha_{0}}u_{0}\|_{L^{p_{0}}_{|x|}L^{\p_{0}}_{\theta}}
\end{equation}
provided that 
\begin{equation}\label{I2Loc}
\quad p_{0} \leq q < \frac{np_{0}}{(\alpha_{0} - \beta)p +n-2}, \quad \p_{0} \leq q, \quad \Lambda(\alpha_{0}, p_{0}, \p_{0}) \geq 0.
\end{equation}
We use Proposition \ref{DDecayCor} and scaling
to bound
\begin{equation}\nonumber
II \leq   d_{0} \| |x|^{\alpha} u \|^{2}_{L^{s}_{T}\Lpptilde}, 
\end{equation}
provided that
\begin{equation}\label{OurYZCond2ProofLoc}
2 \Lambda(\alpha,p, \p) \geq \beta,
\end{equation}

\begin{equation}\label{OurYZCond3ProofLoc}
2 \leq p \leq +\infty, 
\quad 2 < s < +\infty, 
\quad p/2, \ \p /2 \leq q, 
\quad s/2 < r.
\end{equation}
Condition (\ref{OurYZCond2ProofLoc}) is ensured by
\begin{equation}\label{pLocInProofLoc}
\p \geq \frac{2(n-1)}{(2\alpha -\beta)p + 2(n-1)}.
\end{equation}
Then we need a triple $(\beta,r,q)$ such that (\ref{OurYZCond3ProofLoc}) is consistent with $\frac{2}{r} + \frac{n}{q} = 1 -\beta$. 
We are using Theorem \ref{YZTheoremSS} so it is necessary to restrict to $-1\leq \beta$ and, in order to get the lowest value for $\p$, we choose $\beta =-1$. In such a way (\ref{pLocInProofLoc}) becomes (\ref{OurYZCondition3Loc}). By this choice we have 
$$
\p \leq p \quad \mbox{if} \quad -1/2 \leq \alpha,
$$
that is in fact the range of $\alpha$ we have restricted on.
Then we choose $q=p/2$ so by the scaling relation
$$
\frac{2}{r} - \frac{4}{s} = 2 \alpha - 1  \leq 0,
$$
that is consistent with $s/2 < r$.
Because of the choice $q= p/2, \beta = -1$ and the scaling we have to require
$p > n$.
Then (\ref{I2Loc}) becomes  
$$
\quad p_{0} \leq q < \frac{np_{0}}{2p_{0} - 2},
$$ 
that by a straightforward calculation leads to (\ref{OurYZCond0ComplicataLoc}) and $\alpha_{0} \in \left[ 1-n, \frac{2-n}{2+n}  \right)$.

\subsection*{Case $\alpha \in [0,1)$} 
The only difference is again in the choice of $(\beta, r, q)$. 
Since $\alpha \geq 0$ we can reach smaller values for $\p$ by setting 
$2\alpha - \beta = 1 - \varepsilon$ in (\ref{pLocInProofLoc}),
with $\varepsilon > 0$ arbitrarily small. In such a way
$$
\p \geq \frac{2(n-1)p}{(1 - \varepsilon)p  + 2(n-1)}.
$$ 
Then we choose 
$$
(\beta, r, q) = \left( 2 \alpha - 1 + \varepsilon, 
\frac{2s}{4 - \varepsilon s}, 
\frac{p}{2}\right).
$$
It is easy to check that this is consistent with the scaling relation. 
Notice also that the scaling assumptions force to be
$p > n/(1-\alpha)$.
Finally, by (\ref{I2Loc}) and scaling we have
$$
p_{0} \leq q < \frac{np_{0}}{(2 - 2\alpha)p_{0} -2},
$$
that by a straightforward calculation leads to (\ref{OurYZCond0ComplicataBisLoc}) and $$\alpha_{0} \in \left[ 1-(1-\alpha)n, 1-(1-\alpha)\frac{2n}{2+n} \right).$$   
\end{proof}

\section{Outlooks and remarks}
 
In this paper we develop a technique that makes able to 
get new regularity criteria for weak solutions of (\ref{CauchyNS}) from  
a given one.
In principle this machinery could be applied to many different criteria known in 
literature even if we basically focus on (\ref{SerrinNorm}) 
and on Theorem \ref{YZTheoremSS}.

The relations between the indexes in the main theorems are not the most general possible, for instance different choices are allowed than $q=p/2$ or $q=\p_{G}/2$ in the proofs. Anyway we prefer to lose a little in generality in order to get simpler statements.

In the second section we prove time decay estimates for the heat and Oseen kernels that
we consider of independent interest. In particular we plan to use them to study the small data problem for (\ref{CauchyNS}) in future works.


\section{Acknowledgements}

The author would like to thank professor Piero D'Ancona for constant help and suggestions and professor Keith Rogers for useful discussions and reading the paper.


\begin{thebibliography}{10}


\bibitem{DanconaCacciafesta11-a}
F. Cacciafesta and P. D'Ancona.
\newblock Endpoint estimates and global existence for the nonlinear Dirac equation with potential.
\newblock  {\em J. Diff. Eq.}, 254(5):2233--2260, 2013. 




\bibitem{CKN} 
L. Caffarelli, R. Kohn and L. Nirenberg. 
\newblock Partial regularity of suitable weak solutions of the Navier--Stokes equations. 
\newblock {\em Comm. Pure Appl. Math.}, 35:771--831, 1982.



\bibitem{ChoOzawa09-a}
Y. Cho and T. Ozawa.
\newblock Sobolev inequalities with symmetry.
\newblock {\em Commun. Contemp. Math.}, 11(3):355--365, 2009.



\bibitem{DL} 
P. D'Ancona and R. Luc\`a.
\newblock Stein-Weiss and Caffarelli-Kohn-Nirenberg inequalities with higher
angular integrability.
\newblock {\em J. Math. Anal. App.}, 388(2):1061--1079, 2012.



\bibitem{DenapoliDrelichmanDuran09-a}
P.~L. De~N{\'a}poli, I. Drelichman and R.~G. Dur{\'a}n.
\newblock Radial solutions for {H}amiltonian elliptic systems with weights.
\newblock {\em Adv. Nonlinear Stud.}, 9(3):579--593, 2009.


\bibitem{DenapoliDrelichman09-aBis}
 P.~L. De~N{\'a}poli and I. Drelichman.
\newblock Weighted convolution inequalities for radial functions. 
\newblock {\em Ann. Mat. Pur. Appl.}, to appear. 




\bibitem{DenapoliDrelichmanDuran10-a}
 P.~L. De~N{\'a}poli, I. Drelichman and R.~G. Dur{\'a}n.
\newblock Improved Caffarelli-Kohn-Nirenberg and trace inequalities for radial functions. 
\newblock {\em  Comm. Pure Appl. Anal.}, 11(5):1629--1642, 2012.  


\bibitem{DenapoliDrelichmanDuran11-a}
P.~L. De~N{\'a}poli, I. Drelichman and R.~G. Dur{\'a}n.
\newblock On weighted inequalities for fractional integrals of radial functions.
\newblock {\em Illinois J. Math.}, 55:575--587, 2011.



\bibitem{Esc}
L. Escauriaza, G. Seregin and V. Sverak
\newblock Backward uniqueness for parabolic equations.
\newblock {\em Arch. Ration. Mech. Anal.}, 169:147--157, 2003.




\bibitem{FangWang08-a}
D. Fang and C. Wang.
\newblock Weighted {S}trichartz estimates with angular regularity and their applications.
\newblock {\em Forum Math.}, 23:181--205, 2011




\bibitem{Giga} 
Y. Giga. 
\newblock Solutions for semilinear parabolic equations in $L^{p}$ and regularity of weak solutions of the Navier--Stokes system. 
\newblock {\em J. Diff. Eq.}, 62:186--212, 1986.




\bibitem{Hopf} 
E. Hopf. 
\newblock Uber die Anfanqswertaufgabe f\"{u}r die hydrodynamischen Grundgleichungen. 
\newblock {\em Math. Nachr.}, 4:213--231, 1951.






\bibitem{Kat}
T. Kato. 
\newblock Strong $L^{p}$-solutions of the Navier--Stokes equation in $\mathbb{R}^{n}$, with applications to weak solutions. 
\newblock {\em Math. Z.}, 187: 471--480, 1984.




\bibitem{Kuk1}
I. Kukavica.
\newblock On partial regularity for the Navier--Stokes equations.
\newblock {\em Discrete Contin. Dyn. Syst.}, 21:717--728, 2008.


\bibitem{Kuk2}
I. Kukavica.
\newblock On regularity for the Navier--Stokes equations in Morrey spaces. 
\newblock {\ Discrete Contin. Dyn. Syst.} 26(4):1319--1328, 2010. 





\bibitem{Lem}
P. G. Lemari{\'e}-Rieusset.
\newblock Recent developments in the Navier--Stokes problem.
\newblock CHAPMAN AND HALL/CRC. Research Notes in Mathematics Series 431, 2002.




\bibitem{Ler} 
J. Leray. 
\newblock Sur le mouvement d'un liquide visqueux emplissant l'espace. 
\newblock {\em Acta Math.}, 63:193--248, 1934.





\bibitem{MachiharaNakamuraNakanishi05-a}
S. Machihara, M. Nakamura, K. Nakanishi, and T. Ozawa.
\newblock Endpoint {S}trichartz estimates and global solutions for the
  nonlinear {D}irac equation.
\newblock {\em J. Funct. Anal.}, 219(1):1--20, 2005.




\bibitem{Rogers}
T. Ozawa and K. M. Rogers
\newblock Sharp Morawetz estimates.
\newblock {\em J. Anal. Math.}, 121:163--175, 2013.





\bibitem{Rubin83-a}
B. Rubin.
\newblock One-dimensional representation, inversion and certain 
properties of Riesz potentials of radial functions. (russian).
\newblock {\em Mat. Zametki}, 34(4):521---533, 1983.
\newblock English translation: Math. Notes 34(3-4):751--757, 1983.



 






\bibitem{Ser} 
J. Serrin. 
\newblock On the interior regularity of weak solutions of the Navier--Stokes equations. 
\newblock {\em Arch. Ration. Mech. Anal.}, 9:187--195, 1962.  

\bibitem{Sohr} 
H. Sohr. 
\newblock Zur Regularit\"{a}tstheorie der instationaren Gleichungen von Navier--Stokes.
\newblock {\em Math. Z.}, 184:339--375, 1983.  


\bibitem{Stein}
E. M. Stein.
\newblock Note on singular integrals.
\newblock {\em Proc. Am. Math. Soc.}, 8:250--254, 1957.


\bibitem{Struwe} 
M. Struwe.
\newblock On partial regularity results for the Navier--Stokes equations.
\newblock {\em Commun. Pure Appl. Math.} 41:437--458, 1988.

\bibitem{Tay}
M. E. Taylor 
\newblock Analysis on Morrey spaces and applications to Navier--Stokes and other evolution equations. 
\newblock {\ Comm. Part. Diff. Eq.}, 17(9-10):1407--1456, 1992. 



\bibitem{YongZhou} 
Y. Zhou. 
\newblock Weighted regularity criteria for the three-dimensional Navier--Stokes equations. \newblock {\em Proc. Royal Soc. Edin.}, 139A:661--671, 2009.













\bibitem{Sterbenz05-a}
J. Sterbenz.
\newblock Angular regularity and {S}trichartz estimates for the wave equation.
\newblock {\em Int. Math. Res. Not.}, (4):187--231, 2005.
\newblock With an appendix by Igor Rodnianski.




\bibitem{WW}
W. von Wahl. 
\newblock Regularity of weak solutions of the Navier--Stokes equations. In Proc. 1983
Summer Inst. on Nonlinear Functional Analysis and Applications, Proceedings of Symposia
in Pure Mathematics, vol. 45, pp. 497--503 (Providence, RI: American Mathematical Society,
1989).



\end{thebibliography}
\end{document}